\documentclass[11pt,leqno]{amsart}

\usepackage{latexsym,amsmath,amsfonts,amscd}

\newtheorem{theorem}{Theorem}[section]
\newtheorem{lemma}[theorem]{Lemma}
\newtheorem{proposition}[theorem]{Proposition}
\newtheorem{corollary}[theorem]{Corollary}

\newtheorem{remark}[theorem]{Remark}

\def\r{\mathbb{R}}
\def\rn{\mathbb{R}^N}

\def\zn{\mathbb{Z}^N}

\def\eps{\varepsilon}
\def\rh{\rightharpoonup}
\def\irn{\int_{\rn}}

\def\io{\int_{\Omega}}

\def\wt{\widetilde}
\def\wh{\widehat}

\def\la{\langle}
\def\ra{\rangle}
\def\lam{\lambda}
\def\lam1{\lambda_1}
\def\d12{\mathcal{D}^{1,2}}
\def\cm{\mathcal{M}}
\def\meas{\text{meas}}

\numberwithin{equation}{section}

\title[Generalized Nehari manifold]{Generalized Nehari manifold and semilinear Schr\"odinger equation with weak monotonicity condition on the nonlinear term}

\author{Francisco Odair de Paiva}\thanks{The first author was supported by FAPESP under the grant 2015/10545-0. This work was done while he was visiting the Mathematics Department of Stockholm University. He would like to thank the members of the Department for hospitality and a stimulating scientific atmosphere. }
\address{Universidade Federal de S\~ao Carlos, Departamento de Matem\'atica, 13565-905 S\~ao Carlos, Brazil}
\email{odair@dm.ufscar.br}

\author{Wojciech Kryszewski}\thanks{The second author was partially supported by the Polish National Science Center
under grant 2013/09/B/ST1/01963.
}
\address{Faculty of Mathematics and Computer Sciences, Nicolaus Copernicus University, Chopina 12/18, 87-100 Toru\'n, Poland}
\email{wkrysz@mat.umk.pl}

\author{Andrzej Szulkin} 
\address{Department of Mathematics, Stockholm University, 106 91  Stockholm, Sweden}
\email{andrzejs@math.su.se}



\begin{document}

\baselineskip15pt

\maketitle

\begin{abstract}
We study the Schr\"odinger equations $-\Delta u + V(x)u = f(x,u)$ in $\rn$ and $-\Delta u - \lambda u = f(x,u)$ in a bounded domain $\Omega\subset\rn$. We assume that $f$ is superlinear but of subcritical growth and $u\mapsto f(x,u)/|u|$ is nondecreasing. In $\rn$ we also assume that $V$ and $f$ are periodic in $x_1,\ldots,x_N$. We show that these equations have a ground state and that there exist infinitely many solutions if $f$ is odd in $u$. Our results generalize those in \cite{sw1} where $u\mapsto f(x,u)/|u|$ was assumed to be strictly increasing. This seemingly small change forces us to go beyond methods of smooth analysis.
\end{abstract}

\section{introduction} \label{intro}

We consider the semilinear Schr\"odinger equations
\begin{equation} \label{se}
-\Delta u + V(x)u = f(x,u), \quad u\in H^1(\rn)
\end{equation}
and
\begin{equation} \label{se1}
-\Delta u - \lambda u = f(x,u), \quad u\in H^1_0(\Omega),
\end{equation}
where $\Omega\subset\rn$ is a bounded domain and $H^1(\rn)$, $H_0^1(\Omega)$ are the usual Sobolev spaces. In both problems we make the following assumptions on $f$:
\begin{itemize}
\item[$(F_1)$] $f$ is continuous and $|f(x,u)| \le C(1+|u|^{p-1})$ for some $C>0$ and $p\in(2,2^*)$, where $2^*:=2N/(N-2)$ if $N\ge 3$ and $2^*:=+\infty$ if $N=1$ or 2,
\item[$(F_2)$] $f(x,u)=o(u)$ uniformly in $x$ as $u \to 0$,
\item[$(F_3)$] $F(x,u)/u^2 \to\infty$ uniformly in $x$ as $|u|\to\infty$, where $F(x,u) := \int_0^uf(x,s)\,ds$,
\item[$(F_4)$] $u \mapsto f(x,u)/|u|$ is non-decreasing on $(-\infty,0)$ and on $(0,\infty)$.
\end{itemize}
The assumptions $(F_1)$--$(F_3)$ appear in \cite{sw1} while a condition corresponding to $(F_4)$ is a little stronger there: 
\begin{itemize}
\item[$(F_4')$] $u \mapsto f(x,u)/|u|$ is strictly increasing on $(-\infty,0)$ and on $(0,\infty)$.
\end{itemize}
 As we shall see, this slightly weaker hypothesis will force us to go beyond methods of smooth analysis, and introducing a non-smooth approach in this context is in fact our main purpose. In what follows we shall frequently refer to different results and arguments in \cite{sw1, sw2}. When such reference is made, it should be understood that no stronger conditions than $(F_1)$--$(F_4)$ were needed there. 
 
 The main results of this paper are the following two theorems:

\begin{theorem} \label{thm1}
Suppose $f$ satisfies $(F_1)$--$(F_4)$, $V$ and $f$ are 1-periodic in $x_1,\ldots,x_N$ and $0\notin \sigma(-\Delta+V)$, where $\sigma(\cdot)$ denotes the spectrum in $L^2(\rn)$. Then equation \eqref{se} has a ground state solution. If moreover $f$ is odd in $u$, then equation \eqref{se} has infinitely many pairs of geometrically distinct solutions. 
\end{theorem}

\begin{theorem} \label{thm2} 
(i) Suppose $f$ satisfies $(F_1)$--$(F_4)$ and $\lambda\ne\lambda_k$ for any $k$, where $\lambda_k$ is the $k$-th eigenvalue of $-\Delta$ in $H^1_0(\Omega)$. Then equation \eqref{se1} has a ground state solution. If moreover $f$ is odd in $u$, then equation \eqref{se} has infinitely many pairs of geometrically distinct solutions $\pm u_k$ such that the $L^\infty(\Omega)$-norm of $u_k$ tends to infinity with $k$. \\
(ii) If $\lambda=\lambda_k$ for some $k$, then the above results remain valid under the additional assumption that $f(x,u)\ne 0$ unless $u=0$. 
\end{theorem}

Similar results, but under the stronger condition $(F_4')$, have been proved in \cite{sw1}.

As usual, a \emph{ground state} is a solution which minimizes the functional corresponding to the problem over the set of  all nontrivial ($u\ne0$) solutions. Later in this section we shall define what we mean by geometrically distinct solutions.  

Existence of a ground state solution under the assumptions of Theorem \ref{thm1} has been shown by S. Liu in \cite{liu}; since this result is an easy consequence of our approach, we include it here anyway. See also \cite{zz} where a number of results on ground states for problems similar to \eqref{se} and \eqref{se1} has been proved and \cite{ta1} where $(F_4)$ has been further weakened. Existence of ground states for \emph{systems} of equations has been discussed in \cite{me}. Concerning existence of infinitely many solutions we know of a result  by Tang \cite{ta2} where a condition different from $(F_4)$ has been introduced for \eqref{se1}, and by Zhong and Zou \cite{zz} where \eqref{se} and \eqref{se1} have been considered under the same hypotheses as in Theorems \ref{thm1} and \ref{thm2}. However, they needed an additional assumption which is not easy to verify unless $u \mapsto f(x,u)/|u|$ is ``most times'' strictly increasing. 

Consider equation \eqref{se} under the assumptions of Theorem \ref{thm1}. Let $E := H^1(\rn)$. The functional corresponding to \eqref{se} is 
\[
\Phi(u) := \frac12 \irn(|\nabla u|^2+V(x)u^2)\,dx - \irn F(x,u)\,dx.
\]
It is well known (see e.g.\ \cite{wi}) that $\Phi\in C^1(E,\r)$ and critical points of $\Phi$ are solutions for \eqref{se}.
Let $E=E^+\oplus E^-$ be the decomposition corresponding to the positive and the negative part of the spectrum of $-\Delta+V$. Since $0\notin \sigma(-\Delta+V)$, there exists an equivalent inner product $\langle.\,,.\rangle$ in $E$ such that
\begin{equation} \label{fcl}
\Phi(u) = \frac12\|u^+\|^2-\frac12\|u^-\|^2 - \irn F(x,u)\,dx,
\end{equation}
where $u^\pm\in E^\pm$. 

For equation \eqref{se1} under the assumptions of Theorem \ref{thm2} we put $E=H^1_0(\Omega)$ and we have the spectral decomposition $E = E^+\oplus E^0\oplus E^-$, where $E^0$ is the nullspace of $-\Delta-\lambda$ in $E$ and $0 \le \dim (E^0\oplus E^-) < \infty$. Also here we can choose an equivalent inner product such that the corresponding functional $\Phi$ is of the form \eqref{fcl}, with $\rn$ replaced by $\Omega$.

 The following set introduced by Pankov \cite{pa} is called the \emph{generalized Nehari manifold} or the \emph{Nehari-Pankov manifold}:
\begin{equation}
  \label{pan}
\cm := \left\{u\in E\setminus (E^0\oplus E^-):  \Phi'(u)u = 0 \text{ and } \Phi'(u)v = 0 \text{ for all } v\in E^0\oplus E^-\right\}
\end{equation}
($E^0$ is necessarily trivial in Theorem \ref{thm1}).
$(F_4)$ implies $f(x,u)u\ge 0$, and the assumptions of Theorem \ref{thm2} imply that if $\dim E^0>0$, then $f(x,u)u>0$ for $u\ne 0$. Hence $\cm$ contains all nontrivial critical points of $\Phi$. Note that if $E^0\oplus E^-=\{0\}$, then $\cm$ is the usual Nehari manifold \cite{sw2}. Since this case is considerably easier to handle, we assume in what follows that $\sigma(-\Delta+V)\cap(-\infty,0)\ne \emptyset$ in Theorem \ref{thm1} and $\lambda\ge \lambda_1$ in Theorem \ref{thm2}. As in \cite{sw1}, for $u\notin E^0\oplus E^-$ we define
\begin{align}
  \label{eu}
E(u) & :=  E^0\oplus E^-\oplus \r u = E^0\oplus E^-\oplus \r u^+ \\
& \quad \text{and} \quad  \wh E(u) := E^0\oplus E^-\oplus \r^+ u = E^0\oplus E^-\oplus \r^+ u^+, \nonumber
\end{align}
where $\r^+=[0,\infty)$. It has been shown there that if $(F_4)$ is replaced by $(F_4')$, then $\wh E(u)$ intersects $\cm$ at a unique point which is the unique global maximum of $\Phi|_{\wh E(u)}$. It has been shown in \cite{zz} by an explicit example that if $(F_4)$ but not $(F_4')$ holds, then (in the framework of Theorem \ref{thm2}) $\wh E(u)$ and $\cm$ may intersect on a finite line segment. In the next section we shall show that $\wh E(u)\cap\cm \ne\emptyset$ and if $w\in \wh E(u)\cap\cm$, then there exist $\sigma_w > 0$, $\tau_w\ge \sigma_w$ such that $\wh E(u)\cap\cm = [\sigma_w,\tau_w]w$. In other words, $\wh E(u)\cap \cm$ is either a point or a finite line segment. We also show that a point $\wt w\in[\sigma_w,\tau_w]w$ is critical for $\Phi$ if and only if the whole segment $[\sigma_w,\tau_w]w$ consists of critical points.

In Theorem \ref{thm1} the functional $\Phi$ is invariant with respect to the action of $\zn$ given by the translations $k\mapsto u(\cdot-k)$, $k\in\zn$. Hence if $u$ is a solution of \eqref{se}, then so is $u(\cdot-k)$. This and the preceding paragraph justify the following definition: Two solutions $u_1$ and $u_2$ are called \emph{geometrically distinct} if $u_2\ne u_1(\cdot-k)$ for any $k\in\zn$ and $u_2\notin [\sigma_{u_1},\tau_{u_1}]u_1$. In Theorem \ref{thm2} there is no $\zn$-invariance but we still want to identify solutions in  $\wh E(u)\cap\cm$. So $u_1,u_2$ are \emph{geometrically distinct} if $u_2\notin[\sigma_{u_1},\tau_{u_1}]u_1$.  

\medskip

\section{Preliminaries} \label{prel}

In this section we assume that the hypotheses of Theorem \ref{thm1} or \ref{thm2} are satisfied. In particular, $(F_1)$--$(F_4)$ hold. To simplify notation, $\Omega$ will stand for $\rn$ or for a bounded domain in $\rn$.    

\begin{lemma} \label{ineq}
If $f(x,u)\ne 0$, then $F(x,u)<\frac12f(x,u)u$.
\end{lemma}

\begin{proof}
Suppose $u>0$. Since $f(x,t)/t\to 0$ as $t\to 0$ and $f(x,u)/u>0$,
\[
F(x,u) = \int_0^u\frac{f(x,t)}t\,t\,dt < \frac{f(x,u)}u\int_0^ut\,dt = \frac12f(x,u)u.
\]
For $u<0$ the proof is similar.
\end{proof}

The following result will be crucial for studying the structure of the set $\wh E(u)\cap\cm$.  

\begin{proposition} \label{prop}
Let $x\in\Omega$ be fixed and let $u,s,v\in\r$ be such that $s\ge 0$ and $f(x,u)\ne 0$. Then: \\
(i)
\begin{equation} \label{eq1}
g(s,v) := f(x,u)\left[\textstyle{\frac12}\left(s^2-1\right)u+sv\right] + F(x,u)-F(x,su+v) \le 0
\end{equation}
for all $x$. \\
(ii) There exist $s_u\in(0,1]$, $t_u\ge1$ such that $g(s,v)=0$ if and only if $s\in[s_u,t_u]$ and $v=0$ ($s_u=t_u$ not excluded). Moreover, for such $s$ we have $f(x,su) = sf(x,u)$.
\end{proposition}

Part (i) of this proposition has been shown in \cite{liu} and it extends a similar result in \cite{sw1} where $(F_4')$ has been assumed (however, our $s$ corresponds to $s+1$ in \cite{liu, sw1}). Here we provide a different argument which will be needed in order to show part (ii). 

\begin{proof}
Obviously, $g(1,0)=0$. We shall show that $g(s,v)\to-\infty$ as $s+|v|\to \infty$. Put $z=z(s):=su+v$. Using Lemma \ref{ineq}, we obtain
\begin{align*}
g(s,v) & = f(x,u)\left[\textstyle{\frac12}\left(s^2-1\right)u+sv\right]+F(x,u)-F(x,z) \\
&  < f(x,u)\left[\textstyle{\frac12}\left(s^2-1\right)u+s(z-su)\right]+\frac12f(x,u)u-F(x,z) \\
& = -\textstyle{\frac12}s^2f(x,u)u + sf(x,u)z -Az^2 + (Az^2 - F(x,z)). 
\end{align*}
Since the quadratic form (in $s$ and $z$) above is negative definite if $A>0$ is a constant large enough and since $Az^2 - F(x,z)$ is bounded above according to $(F_3)$, $g(s,v)\to-\infty$ as $s+|v|\to \infty$ as claimed. 

It follows that $g$ has a maximum $\ge 0$ on the set $\{(s,v): s\ge 0\}$. As 
\[
g(0,v) = -\textstyle{\frac12}f(x,u)u+F(x,u)-F(x,v) < -F(x,v) \le 0
\]
(by Lemma \ref{ineq}), the maximum is attained at some $(s,v)$ with $s>0$. Then 
\begin{equation} \label{e1}
g'_v(s,v) = sf(x,u)-f(x,su+v) = 0
\end{equation}
and
\begin{equation} \label{e2}
g'_s(s,v) =(su+v)f(x,u)-uf(x,su+v) = 0.
\end{equation} 
Using \eqref{e1} in \eqref{e2} we obtain $vf(x,u)=0$. Hence $v=0$ and 
\[
g'_s(s,0) = su^2\left(\frac{f(x,u)}{u} - \frac{f(x,su)}{su}\right) = 0.
\]
By $(F_4)$,  there must exist $s_u, t_u$ such that $s_u\in(0,1]$, $t_u\ge1$ and $g'_s(s,0)=0$ if and only if $s\in[s_u,t_u]$. For such $s$ we have $g(s,0)=g(1,0)=0$ and $f(x,su) = sf(x,u)$. 
\end{proof}

\begin{corollary} \label{cor}
Suppose $u\in\cm$ and let $s\ge 0$, $v\in E^0\oplus E^-$. Then
\[
\io \left(f(x,u)\left[\textstyle{\frac12}\left(s^2-1\right)u+sv\right] + F(x,u)-F(x,su+v)\right) dx \le 0
\]
and there exist $0<s_u\le 1\le t_u$ such that equality holds if and only if $s\in[s_u,t_u]$, $v=0$. Moreover, for such $s$ and almost all $x\in\Omega$, $f(x,su) = sf(x,u)$.
\end{corollary}

\begin{proof}
If $u\in\cm$, then $f(x,u(x))\ne 0$ for $x$ on a set of positive measure. According to Proposition \ref{prop}, inequality \eqref{eq1} holds for such $x$ and there exist $s_{u(x)}\in(0,1]$, $t_{u(x)}\ge 1$ such that the left-hand side of \eqref{eq1} is zero if and only if $s\in[s_{u(x)},t_{u(x)}]$ and $v(x)=0$. Moreover, for such $s$, $f(x,su(x)) = sf(x,u(x))$. Now one takes $s_u := \text{ess}\sup\{s_{u(x)}: f(x,u(x))\ne 0\}$ and $t_u := \text{ess}\inf\{t_{u(x)}: f(x,u(x))\ne 0\}$. 

Note that if $f(x,u(x))=0$, then $F(x,u(x)) = \int_0^{u(x)}f(x,t)\,dt = 0$ because $f(x,t)=0$ for $t$ between 0 and $u(x)$ according to $(F_4)$. Hence the integrand above is $\le 0$ also in this case.  
\end{proof}

\begin{proposition} \label{mainprop}
(i) If $u\in E\setminus (E^0\oplus E^-)$, then $\wh E(u)\cap\cm\ne\emptyset$. \\ 
(ii) If $w\in \wh E(u)\cap\cm$, then there exist $0<s_w\le 1 \le t_w$ such that $\wh E(u)\cap\cm = [s_w,t_w]w$. Moreover, $\Phi(sw)=\Phi(w)$, $\Phi'(sw) = s\Phi'(w)$ for all $s\in [s_w,t_w]$ and $\Phi(z)<\Phi(w)$ for all other $z\in \wh E(u)$. \\
(iii) $\cm$ is bounded away from $E^0\oplus E^-$, closed and $c := \inf_{w\in\cm}\Phi(w) > 0$. Moreover, $\Phi|_\cm$ is coercive, i.e., $\Phi(u)\to\infty$ as $u\in\cm$ and $\|u\|\to\infty$.
\end{proposition}

Note that an immediate consequence is that if $w$ is a critical point of $\Phi$, then the whole line segment $ [s_w,t_w]w$ consists of critical points. 

\begin{proof}
(i) The conclusion can be found in \cite[Lemma 2.6 and Theorem 3.1]{sw1}, see also \cite[Proposition 39]{sw2}. The proof is by showing that $\Phi(z)\le 0$ for $z\in \wh E(u)$ and $\|z\|$ large enough, and then weak upper semicontinuity of $\Phi|_{\wh E(u)}$ implies that there exists a positive maximum. 

(ii)  For each $z\in \wh E(u)$ we have $z = sw+v$, where $s\ge 0$ and $v=v^0+v^-\in E^0\oplus E^-$. It has been shown in the course of the proof of \cite[Proposition 2.3]{sw1} and \cite[Proposition 39]{sw2} that 
\begin{align*}
& \Phi(z)-\Phi(w) = \Phi(sw+v) -\Phi(w) =  -\frac12\|v^-\|^2  \\
& \qquad + \io \left(f(x,w)\left[\textstyle{\frac12}\left(s^2-1\right)w+sv\right]+F(x,w)-F(x,sw+v)\right) dx
\end{align*}
(again, keep in mind that our $s$ corresponds to $s+1$ in \cite{sw1, sw2}).
Hence according to Corollary \ref{cor}, $\Phi(z)\le\Phi(w)$ for all $z\in \wh E(u)$ and $\Phi(z)=\Phi(w)$ if and only if $z\in  [s_w,t_w]w$. That $\Phi(sw)=\Phi(w)$ for $s\in [s_w, t_w]$ is clear and since $\Phi(sw) = \max_{\wh E(u)}\Phi(z)$, it is also clear that $\wh E(u)\cap\cm = [s_w,t_w]w$ and $\Phi(z)<\Phi(w)$ for other $z$. The equality $\Phi'(sw) = s\Phi'(w)$ follows immediately from the fact that $f(x,sw) = sf(x,w)$.

(iii) That $c>0$ has been shown in \cite[Lemma 2.4]{sw1} and is an immediate consequence of the fact that $\Phi(u) = \frac12\|u\|^2+o(\|u\|^2)$ as $u\to 0$, $u\in E^+$. Since $\Phi|_{E^0\oplus E^-}\le 0$, $\cm$ is bounded away from $E^0\oplus E^-$ and hence closed. Finally, according to Proposition 2.7 and the proof of Theorem 3.1 in \cite{sw1}, $\Phi|_\cm$ is coercive.
\end{proof}

\begin{remark} \label{rem1}
\emph{
If $f$ satisfies $(F_1)$--$(F_4)$ and is of the form $f(x,u) = a(x)h(u)$, where $h(u)\ne 0$ for $u\ne 0$, then $s_w=t_w=1$ in Proposition \ref{mainprop}, i.e.\ $\wh E(u)$ intersects $\cm$ at a unique point. Assuming the contrary, suppose $t_w>1$ and $w>0$ on a set of positive measure (other cases are treated similarly). So $\meas\{x: w(x)>d\}$ is positive for some $d>0$. We claim that $h(t)/t$ is constant for $0<t<d$. Otherwise there exist $s\in(1,t_w]$, $t_0$ and $\eps>0$ such that $\eps<t_0<d-\eps$ and 
\[
\frac{h(t)}t < \frac{h(st)}{st} \quad \text{for all } t\in(t_0-\eps, t_0+\eps).
\]
Since the sets $\{x: w(x)>t_0+\eps\}$ and $\{x:w(x)<t_0-\eps\}$ have positive measure, so does the set $\{x: w(x)\in (t_0-\eps,t_0+\eps)\}$, see \cite{chh}. But this contradicts the last statement of Corollary \ref{cor}. Hence $h(t)/t$ is constant for $0<t<d$ and $h(t)/t\to 0$ as $t\to 0$. So $h(t)=0$ on $(0,d)$ which is impossible. 
}
\end{remark}

According to Proposition \ref{mainprop}, for each $u\in E^+\setminus\{0\}$ there exist $w$ and $0<\sigma_w\le\tau_w$ such that
\[
m(u) := [\sigma_w,\tau_w]w = \wh E(u)\cap\cm\subset E.
\]
This is a multivalued map from $E^+\setminus\{0\}$ to $E$. However, the map $\wh\Psi: E^+\setminus\{0\} \to \r$ given by
\[
\wh\Psi(u) := \Phi(m(u)) = \max_{z\in \wh E(u)}\Phi(z)
\] 
is single-valued because $\Phi$ is constant on $\wh E(u)\cap\cm$. In fact more is true:

\begin{proposition} \label{lip}
The map $\wh\Psi$ is locally Lipschitz continuous. 
\end{proposition}

\begin{proof}
If $u_0\in E^+\setminus\{0\}$, then there exist a neighbourhood $U\subset E^+\setminus\{0\}$ of $u_0$ and $R>0$ such that $\Phi(w)\le 0$ for all $u\in U$ and $w\in \wh E(u)$, $\|w\|\ge R$. For otherwise we can find sequences $(u_n)$, $(w_n)$ such that $u_n\to u_0$, $w_n\in \wh E(u_n)$, $\Phi(w_n)>0$ and $\|w_n\|\to\infty$. But $u_0, u_1, u_2,\ldots$ is a compact set, hence according to \cite[Lemma 2.5]{sw1}, $\Phi(w)\le 0$ for some $R$ and all $w\in\wh E(u_j)$, $j=0,1,2,\ldots$\,, $\|w\|\ge R$, which is a contradiction. 

Let $U,R$ be as above and $s_1u_1+v_1\in m(u_1)$, $s_2u_2+v_2\in m(u_2)$, where $u_1,u_2\in U$ and $v_1,v_2\in E^0\oplus E^-$. Then $\|m(u_1)\|, \|m(u_2)\| \le R$. By the maximality property of $m(u)$ and the mean value theorem, 
\begin{align*}
\wh\Psi(u_1)-\wh\Psi(u_2) & = \Phi(s_1u_1+v_1)-\Phi(s_2u_2+v_2) \le \Phi(s_1u_1+v_1)-\Phi(s_1u_2+v_1) \\
&  \le s_1\sup_{t\in[0,1] } \|\Phi'(s_1(tu_1+(1-t)u_2)+v_1)\|\,\|u_1-u_2\| \le C\|u_1-u_2\|,
\end{align*}
where the constant $C$ depends on $R$ but not on the particular choice of points in $m(u_1)$, $m(u_2)$. Similarly, $\wh\Psi(u_2)-\wh\Psi(u_1) \le C\|u_1-u_2\|$ and the conclusion follows.
\end{proof}

\begin{remark} \label{rem2}
\emph{
It has been shown in \cite{sw1} that if $(F_4')$ holds instead of $(F_4)$, then $\wh\Psi\in C^1(E^+\setminus\{0\},\r)$. An easy inspection of the arguments in \cite{sw1} or \cite{sw2} shows that if for each $u\in E^+\setminus\{0\}$ there exists a unique positive maximum of $\Phi|_{\wh E(u)}$, then $\wh\Psi$ is still of class $C^1$. Hence in particular, if $f$ is as in Remark \ref{rem1}, then the conclusions of Theorems \ref{thm1} and \ref{thm2} hold with the same proofs as in \cite{sw1}.
}
\end{remark}

However, under our assumptions we can in general only assert that $\wh\Psi$ is locally Lipschitz continuous (because $u\mapsto m(u)$ may not be single-valued). Therefore, instead of the derivative of $\wh\Psi$ we shall use Clarke's subdifferential \cite{cl}. The study of minimax methods for differential equations whose associated functional is merely locally Lipschitz continuous has been initiated by Chang in \cite{ch}. We  recall some notions and facts taken from \cite{ch, cl}. They may also be found conveniently collected in Section 7.1 of \cite{cha}. The \emph{generalized directional derivative} of $\wh\Psi$ at $u$ in the direction $v$ is defined by
\[
\wh\Psi^\circ(u;v) := \limsup_{\substack{h\to 0\\ t\downarrow 0}} \frac{\wh\Psi(u+h+tv)-\wh\Psi(u+h)}t.
\]
The function $v\mapsto \wh\Psi^\circ(u;v)$ is convex and its subdifferential $\partial\wh\Psi(u)$ is called the \emph{generalized gradient} (or \emph{Clarke's subdifferential}) of $\wh\Psi$ at $u$, that is,
\begin{equation} \label{subdiff}
\partial\wh\Psi(u) := \{w\in E^+: \wh\Psi^\circ(u;v) \ge \la w,v\ra \text{ for all }v\in E^+\}.
\end{equation} 
In \cite{cha} $E$ is a Banach space and the generalized gradient is in the dual space $E^*$. Since here we work in a Hilbert space, we may assume via duality that  $\partial\wh\Psi(u)$ is a subset of $E$ (or more precisely, of $E^+$). A point $u$ is called a \emph{critical point} of $\wh\Psi$ if $0\in \partial\wh\Psi(u)$, i.e. $\wh\Psi^0(u;v)\ge 0$ for all $v\in E^+$, and a sequence $(u_n)$ is called a \emph{Palais-Smale sequence} for $\wh\Psi$ (PS-sequence for short) if $\wh\Psi(u_n)$ is bounded and there exist $w_n\in\partial\wh\Psi(u_n)$ such that $w_n\to 0$. The functional $\wh\Psi$ satisfies the \emph{PS-condition} if each PS-sequence has a convergent subsequence. Below we collect some notation which we shall need:
\begin{gather*}
S^+ := \{u\in E^+: \|u\|=1\}, \quad T_uS^+ := \{v\in E^+: \la u,v\ra =0\} ,\quad \Psi := \wh\Psi|_{S^+}, \\
\Psi^d := \{u\in S^+: \Psi(u)\le d\}, \quad \Psi_c := \{u\in S^+: \Psi(u)\ge c\}, \quad \Psi_c^d := \Psi_c\cap \Psi^d, \\ 
K := \{u\in S^+: 0\in \partial\wh\Psi(u)\} \quad K_c := \Psi_c^c\cap K, \quad \partial\Psi(u) := \partial\wh\Psi(u),\text{\,where }u\in S^+.
\end{gather*}
Note that the symbol $\partial\Psi(u)$ stands for $\partial\wh\Psi(u)$ when $u$ is restricted to $S^+$. This is in consistence with the notation $\Psi = \wh\Psi|_{S^+}$.
As we shall see in the proof of the next proposition, $\wh\Psi^\circ(u;su)=0$ for all $s\in\r$. Hence $\partial\Psi(u)\subset T_uS^+$. 

\begin{proposition} \label{crit}
(i) $u\in S^+$ is a critical point of $\wh\Psi$ if and only if $m(u)$ consists of critical points of $\Phi$. The corresponding critical values coincide. \\
(ii) $(u_n)\subset S^+$ is a PS-sequence for $\wh\Psi$ if and only if there exist $w_n\in m(u_n)$ such that $(w_n)$ is a PS-sequence for $\Phi$.
\end{proposition}

\begin{proof}
(i) Let $u\in S^+$. We shall show that $\wh\Psi^\circ(u;v) \ge 0$ for all $v\in E^+$  if and only if $m(u)$ consists of critical points. Note first that there exists an orthogonal decomposition $E = E(u)\oplus T_uS^+$, and by the maximizing property of $m(u)$, $\Phi'(w)v=0$ for all $w\in m(u)$ and $v\in E(u)$. Let $s\in\r$ be fixed. Since $\wh\Psi(u)=\wh\Psi(\sigma u)$ for all $\sigma>0$ and $\wh\Psi$ is locally Lipschitz continuous, 
\[
|\wh\Psi(u+h+ t(su))-\wh\Psi(u+h)| = |\wh\Psi((1+ts)u+h)-\wh\Psi((1+ts)(u+h))| \le Ct|s|\|h\|
\]
for $\|h\|$ and $t>0$ small. Hence $\wh\Psi^\circ(u;su)=0$ for all $s\in\r$. So we only need to consider $v\in T_uS^+$. 

Let $s_{u}u + z_{u}$, where $s_u>0$ and $z_u\in E^0\oplus E^-$, denote an (arbitrarily chosen) element of $m(u)$. Then, using the maximizing property of $m(u)$ and the mean value theorem,
\begin{align*}
\wh\Psi(u+h+tv)&-\wh\Psi(u+h) = \Phi(s_{u+h+tv}(u+h+tv) + z_{u+h+tv}) - \Phi(s_{u+h}(u+h) + z_{u+h}) \\
&\le \Phi(s_{u+h+tv}(u+h+tv) + z_{u+h+tv}) - \Phi(s_{u+h+tv}(u+h) + z_{u+h+tv}) \\
&= ts_{u+h+tv}\Phi'(s_{u+h+tv}(u+h+\theta tv) + z_{u+h+tv})v
\end{align*}
for some $\theta\in(0,1)$. Dividing by $t$ and letting $h\to 0$ and $t\downarrow 0$ via subsequences we obtain
\begin{equation} \label{critical}
\wh\Psi^\circ(u;v) \le s^*\Phi'(s^*u+z^*)v,
\end{equation}
 where $s_n := s_{u+h_n+t_nv}\to s^*>0$ and  $z_n :=z_{u+h_n+t_nv}\rh z^*$. This follows because $\cm$ is bounded away from 0 and $\Phi|_\cm$ coercive, hence $s_n$ and $z_n$ must be bounded. We claim that $s^*u+z^*\in\cm$. Indeed, taking subsequences once more, writing $z_n=z_n^0+z_n^-\in E^0\oplus E^-$ and using Fatou's lemma,
\begin{align*}
\wh\Psi(u) & = \lim_{n\to\infty} \wh\Psi(u+h_n+t_nv) = \lim_{n\to\infty} \Phi(s_n(u+h_n+t_nv)+z_n) \\
& = \lim_{n\to\infty} \left(\frac12\|s_n(u+h_n+t_nv)\|^2 - \frac12\|z_n^-\|^2 - \io F(x,s_n(u+h_n+t_nv)+z_n)\,dx \right) \\
& \le \frac12\|s^*u\|^2 -\frac12\|(z^*)^-\|^2 - \io F(x,s^*u+z^*)\,dx \le \wh\Psi(u).
\end{align*}
This implies that $\|z_n\|\to \|z^*\|$ (recall $\dim E^0<\infty$), hence $z_n\to z^*$ and $s_n(u+h_n+t_nv)+z_n\to s^*u+z^*$. As $\cm$ is closed, the claim follows. Since $\wh E(u)\cap\cm$ may be a line segment, it is not sure that $s^*$ and $z^*$ are the same for different $v$. However, if $s_1^*,s_2^*$ and $z_1^*,z_2^*$ correspond to $v_1$ and $v_2$, then by Proposition \ref{mainprop}, $s_1^*u+z_1^* = \tau(s_2^*u+z_2^*)$ and $\Phi'(s_1^*u+z_1^*)v_2 = \tau \Phi'(s_2^*u+z_2^*)v_2$ for some $\tau>0$. Taking this into account, we see from \eqref{critical} that if $y\in\partial\Psi(u)$, then
\begin{equation} \label{tau}
\la y,v\ra \le \wh\Psi^\circ(u;v) \le \tau(v) \Phi'(s^*u+z^*)v,
\end{equation}
where $\tau$ is bounded and bounded away from 0 (by constants independent of $v$). It follows immediately that $u$ is a critical point of $\Psi$ if and only if $m(u)$ consists of critical points of~$\Phi$. 

(ii) The proof is very similar here. We take $y_n\in\partial\Psi(u_n)$ and $w_n\in m(u_n)$. Since $\Phi|_\cm$ is coercive, boundedness of $\Phi(m(u_n))$ implies that $(m(u_n))$ is bounded. As in \eqref{tau}, we see that
\begin{equation} \label{taun}
\la y_n,v\ra \le \wh\Psi^\circ(u_n;v)\le \tau_n(v) \Phi'(w_n)v,
\end{equation}
where $\tau_n$ is bounded and bounded away from 0 because so is $m(u_n)$. So the conclusion follows.
\end{proof}

Note that if $(w_n)\subset (m(u_n))$ is a PS-sequence for $\Phi$, then so is any sequence $(w_n')\subset (m(u_n))$.

Finally for this section we construct a pseudo-gradient vector field $H: S^+\setminus K\to TS^+$ for $\Psi$. For $u\in S^+$, let 
\[
\partial^-\Psi(u) := \left\{p\in\partial\Psi(u): \|p\| = \min_{a\in\partial \Psi(u)} \|a\| \right\} \quad\text{and} \quad  \mu(u) := \inf_{a\in S^+}\{\|\partial^-\Psi(a)\|+ \|u-a\|\}.
\]
Since $\partial\Psi(u)$ is closed and convex, $p$ as above exists and is unique, cf.\ \cite{cha, ch}. Hence 
\[
K = \{u\in S^+: \partial^-\Psi(u) = 0\}.
\]
The map $u\mapsto \|\partial^-\Psi(u)\|$ is lower semicontinuous \cite[Proposition 7.1.1(vi)]{cha} but not continuous in general. The reason for introducing the function $\mu$ is that it regularizes $\|\partial^-\Psi(u)\|$. The idea comes from \cite{ck} where a similar function has been defined.

\begin{lemma} \label{cont}
The function $\mu$ is continuous and $u\in K$ if and only if $\mu(u)=0$.
\end{lemma}
\begin{proof}
 Let $u,v,a\in S^+$. Then
\[
\mu(u) \le \|\partial^-\Psi(a)\| + \|u-a\| \le \|\partial^-\Psi(a)\| + \|v-a\| + \|u-v\|,
\]
and taking the infimum over $a$ on the right-hand side we obtain $\mu(u) \le \mu(v) + \|u-v\|$. Reversing the roles of $u$ and $v$ we see that $|\mu(u)-\mu(v)| \le \|u-v\|$. Hence $\mu$ is (Lipschitz) continuous.

Since $0\le \mu(u)\le \|\partial^-\Psi(u)\|$, it is clear that $\mu(u)=0$ if $u\in K$. Suppose $\mu(u)=0$. Then there exist $a_n$ such that $\partial^-\Psi(a_n)\to 0$ and $a_n\to u$, so $u\in K$ by the lower semicontinuity of $u\mapsto \|\partial^-\Psi(u)\|$. 
\end{proof}

\begin{proposition} \label{pseudogr}
There exists a locally Lipschitz continuous function $H: S^+\setminus K\to TS^+$ such that $\|H(u)\|\le 1$ and $\inf\{\la p, H(u)\ra: p\in\partial\Psi(u)\} > \frac12\mu(u)$ for all $u\in S^+\setminus K$.
If $\Phi$ is even, then $H$ may be chosen to be odd.
\end{proposition}

\begin{proof}
Let $u\in S^+\setminus K$ and put $v_u := \partial^-\Psi(u)/\|\partial^-\Psi(u)\|$. Consider the map
\[
\chi: w\mapsto \inf_{p\in \partial\Psi(w)} \left\la  p,\, v_u-\la v_u,w\ra w\right\ra -\frac12\mu(w), \quad w\in S^+\setminus K
\]
(note that $v_u-\la v_u,w\ra w\in T_w(S^+)$). Since $\partial\Psi(u)$ is convex, $ \inf_{p\in \partial\Psi(u)} \la  p, v_u\ra \ge \|\partial^-\Psi(u)\| \ge \mu(u)$ and therefore $\chi(u)\ge \frac12\mu(u) > 0$. Moreover, since
\[
\inf_{p\in \partial\Psi(w)} \left\la  p,\, v_u-\la v_u,w\ra w\right\ra = -\sup_{p\in \partial\Psi(w)} \left\la  p,\, \la v_u,w\ra w-v_u\right\ra = -\wh\Psi^0(w;\, \la v_u,w\ra w- v_u)
\]
(see Proposition 7.1.1(vii) and property (c) on p.\ 168 in \cite{cha}) and $\wh\Psi^0$ is upper semicontinuous in both arguments \cite[Proposition 7.1.1(vii)]{cha}, $\chi$ is lower semicontinuous. Hence there exists a neighbourhood $U_u$ of $u$ such that $\chi(w)>0$ for all $w\in U$.

The remaining part of the proof is standard. Take a locally finite open refinement $(U_{u_i})_{i\in I}$ of the open cover $(U_u)_{u\in S^+\setminus K}$ and a subordinated locally Lipschitz continuous partition of unity $\{\lambda_i\}_{i\in I}$. Define
\[
H(u) := \sum_{i\in I}\lambda_i(u)v_{u_i}, \quad u\in S^+\setminus K.
\]
It is easy to see that $H$ satisfies the required conclusions.

If $\Phi$ is even, then so is $\Psi$ and we may replace $H(u)$ with $\frac12(H(u)-H(-u))$.
\end{proof}

\section{Proofs of Theorems \ref{thm1} and \ref{thm2}} \label{pfs}

Since the arguments are very similar to those appearing in \cite{sw1, sw2}, we shall describe them rather briefly and concentrate on pointing out the main differences. 

We start with Theorem \ref{thm1}. First we want to show that there exists a minimizer for $\Psi$ on $S^+$. It follows from the results of Section \ref{prel} that 
\[
c := \inf_{w\in\cm}\Phi(w) = \inf_{u\in S^+}\Psi(u) > 0.
\]
According to Ekeland's variational principle \cite{ek}, there exists a sequence $(u_n)\subset S^+$ such that $\Psi(u_n)\to c$ and
\begin{equation} \label{eke}
\Psi(w) \ge \Psi(u_n) - \frac1n\|w-u_n\| \quad \text{for all } w\in S^+.
\end{equation}
For a given $v\in T_{u_n}S^+$, let $z_n(t) := (u_n+tv)/\|u_n+tv\|$.  Since $\|u_n+tv\|-1 = O(t^2)$ as $t\to 0$ and $\wh\Psi(u_n+tv)=\Psi(z_n(t))$, it follows from \eqref{eke} that
\[
\wh\Psi^\circ(u_n;v) \ge \limsup_{t\downarrow 0} \frac{\wh\Psi(u_n+tv)-\wh\Psi(u_n)}t = \limsup_{t\downarrow 0} \frac{\Psi(z_n(t))-\Psi(u_n)}t \ge -\frac1n\|v\|.
\]
Since $m(u_n)$ is bounded by coercivity of $\Phi|_\cm$, the second inequality in \eqref{taun} implies that 
\[
-\frac1n\|v\| \le \wh\Psi^\circ(u_n;v) \le \tau_n(v)\Phi'(w_n)v,
\]
where $w_n\in m(u_n)\subset\cm$ and $\tau_n$ is bounded and bounded away from 0. So recalling $\Phi'(w_n)v=0$ for all $v\in E(w_n)$, it follows that $(w_n)$ is a bounded PS-sequence for $\Phi$. Now we may proceed exactly as in the proof of Theorem 1.1 in \cite{sw1}, pp.\ 3811--3812 (or in the proof of Theorem 40 in \cite{sw2}). More precisely, one shows invoking Lions' lemma \cite[Lemma 1.21]{wi} in a rather standard way that there exists a sequence $(y_n)\subset\rn$ such that
\[
\int_{|x-y_n|<1} w_n^2\,dx \ge \eps \quad \text{for $n$ large enough and some $\eps>0$},
\]
and since $\Phi$ and $\cm$ are invariant by translations $u(\cdot)\mapsto u(\cdot-k)$, $k\in\zn$, we may assume $(y_n)$ is bounded. So passing to a subsequence, $w_n\rh w\ne 0$. This $w$ is a solution and an additional argument shows it is a ground state, see \cite{sw1} or \cite{sw2} for more details. 

\medskip

Suppose now $f$ is odd in $u$ and note that $\Psi$ is even and invariant by translations by elements of $\zn$. To prove that there exist infinitely many geometrically distinct solutions we assume the contrary.  Since to each $[\sigma_w,\tau_w]w\subset \cm$ there corresponds a unique point $u\in S^+$, $K$ consists of finitely many orbits $\mathcal{O}(u) := \{u(\cdot -k): u\in K,\ k\in\zn\}$. We choose a subset $\mathcal{F}\subset K$ such that $\mathcal{F} = -\mathcal{F}$ and each orbit has a unique representative in $\mathcal{F}$. Now an easy inspection shows that Lemmas 2.11 and 2.13 in \cite{sw1} hold, i.e. the mapping $\check m: u\mapsto u^+/\|u^+\|$ from $\cm$ to $S^+$ is Lipschitz continuous and $\kappa: = \inf\{v-w\|: v,v\in K,\ v\ne w\} > 0$. 

\begin{proposition}[Lemma 2.14 in \cite{sw1}] \label{discreteness}
Let $d \ge c$. If $(v_n^1), (v_n^2) \subset \Psi^d$ are two PS-sequences for $\Psi$, then either $\|v_n^1-v_n^2\|
\to 0$ as $n\to\infty$ or $\limsup_{n \to \infty}\|v_n^1-v_n^2\| \ge \rho(d)>0$, where $\rho$ depends on $d$ but not on the particular choice of PS-sequences in $\Psi^d$. 
\end{proposition}

The argument is exactly the same as in \cite{sw1}, taking into account that by Proposition \ref{crit}, to $(v_n^j)\subset\Psi^d$ there correspond PS-sequences $(u_n^j)$ with $u_n^j\in m(v_n^j)$, $j=1,2$. Once $u_n^j$ have been chosen, one follows the lines of \cite{sw1}. 

Let $H$ be the vector field constructed in Proposition \ref{pseudogr} and consider the flow given by
\[
\frac d{dt}\eta(t,w) = -H(\eta(t,w)), \quad \eta(0,w)=w,
\]
defined on the set
\[
\mathcal{G} := \{(t,w): w\in S^+\setminus K,\ T^-(w)<t<T^+(w)\},
\]
where $(T^-(w),T^+(w))$ is the maximal existence time for the trajectory passing through $w$ at $t=0$. 

\begin{proposition}[cf.\ Lemma 2.15 in \cite{sw1}] \label{deformation}
For each $w\in S^+\setminus K$ the limit $\lim_{t\to T^+(w)}\eta(t,w)$ exists and is a critical point of $\Psi$.
\end{proposition}

\begin{proof} 
We adapt the argument in \cite{sw1}. 

If $T^+(w)<\infty$, then for $0\le s<t<T^+(w)$ we have
\[
\|\eta(t,w)-\eta(s,w)\| \le \int_s^t\|H(\eta(\tau,w))\|\,d\tau \le t-s,
\]
hence $\lim_{t\to T^+(w)}\eta(t,w)$ exists and must be a critical point (or the flow can be continued for $t>T^+(w)$). 

Let $T^+(w)=\infty$. It suffices to show that for each $\eps>0$ there exists $t_\eps>0$ such that $\|\eta(t_\eps,u)-\eta(t,u)\| < \eps$ for all $t\ge t_\eps.$ Assuming the contrary, we find $\eps\in (0,\rho(d)/2)$ and $t_n\to\infty$ such that $\|\eta(t_n,w)-\eta(t_{n+1},w)\| = \eps$ for all $n$. Choose the smallest $t_n^1\in (t_n,t_{n+1})$ such that $\|\eta(t_n,w)-\eta(t_n^1,w)\|=\eps/3$. 
Recall from Lemma \ref{cont} that $\mu$ is continuous and set $\kappa_n := \min_{s\in[t_n,t_n^1]}\mu(\eta(s,w))$. Then, using Proposition \ref{pseudogr} and \cite[Proposition 7.1.1(viii)]{cha},
\begin{align*}
\frac{\eps}3 & = \|\eta(t_n^1,w)-\eta(t_n,w)\| \le \int_{t_n}^{t_n^1}\|H(\eta(s,w))\|\,ds \le t_n^1-t_n \\
& \le  \frac2{\kappa_n}\int_{t_n}^{t_n^1} \inf_{p\in\partial\Psi(\eta(s,u))}\la p, H(\eta(s,w))\ra\,ds = -\frac2{\kappa_n}\int_{t_n}^{t_n^1} \sup_{p\in\partial\Psi(\eta(s,u))}\la p, -H(\eta(s,w))\ra\,ds \\
&  \le -\frac2{\kappa_n}\int_{t_n}^{t_n^1}\frac d{ds}\Psi(\eta(s,w))\,ds  = \frac2{\kappa_n}(\Psi(\eta(t_n,w)) -\Psi(\eta(t_n^1,w))).
\end{align*}
Since $\Psi$ is bounded below, $\Psi(\eta(t_n,w)) -\Psi(\eta(t_n^1,w))\to 0$, hence $\kappa_n\to 0$ and we may find $s_n^1\in[t_n,t_n^1]$ such that if $z_n^1:=\eta(s_n^1,w)$, then $\mu(z_n^1)\to 0$. By the definition of $\mu$ there exist $w_n^1$ such that $w_n^1-z_n^1\to 0$ and $\partial^-\Psi(w_n^1)\to 0$. So $\limsup_{n\to\infty}\|w_n^1-\eta(t_n,w)\| \le\eps/3$. Similarly, there exists a largest $t_n^2\in(t_n^1,t_{n+1})$ with $\|\eta(t_{n+1},w)-\eta(t_n^2,w)\| = \eps/3$ and we find $w_n^2$ with $\partial^-\Psi(w_n^2)\to 0$ and $\limsup_{n\to\infty}\|w_n^2-\eta(t_{n+1},w)\| \le \eps/3$. It follows that $\eps/3 \le\limsup_{n\to\infty}\|w_n^1-w_n^2\|\le 2\eps<\rho(d)$, a contradiction to Proposition \ref{discreteness}.
\end{proof}

\begin{proposition}[cf.\ Lemma 2.16 in \cite{sw1}] \label{deff}
Let $d\ge c$. Then for each $\delta>0$ there exists $\eps>0$ such that  $\Psi_{d-\eps}^{d+\eps}\cap K = K_d$ and $\lim_{t\to T^+(w)}\Psi(\eta(t,w)) < d-\eps$ for all $w\in \Psi^{d+\eps}\setminus U_\delta(K_d)$, where $U_\delta(K_d)$ is the open $\delta$-neighbourhood of $K_d$.
\end{proposition}

The proof requires changes which, in view of the arguments of Proposition \ref{deformation}, are rather obvious (in particular, $\nabla\Psi$ in the definition of $\tau$ in \cite{sw1} should be replaced by $\mu$). 

With all these prerequisites, existence of infinitely many solutions is obtained by repeating the arguments on pp.\ 3817--3818 in \cite{sw1}. Let
\[
c_k := \inf \{d\in\r: \gamma(\Psi^d)\ge k\}, \quad k=1,2,\ldots,
\] 
where $\gamma$ denotes Krasnoselskii's genus \cite{st}. Using the flow $\eta$ and Proposition \ref{deff} one shows $K_{c_k}\ne 0$ and $c_k<c_{k+1}$ for all $k$. This contradicts our assumption that there are finitely many geometrically distinct solutions. 

\medskip	

Now we turn our attention to Theorem \ref{thm2}. Here there is no $\zn$-symmetry but instead there is a compact embedding $H^1_0(\Omega) \hookrightarrow L^q(\Omega)$ for $q\in[1,2^*)$. Using this, one sees as in the proof of Theorem 3.1 in \cite{sw1} or Theorem 37 in \cite{sw2} that the ground state exists. The minimizing PS-sequence is extracted by using Ekeland's variational principle in the same way as at the beginning of this section. To obtain infinitely many solutions for odd $f$ one first shows as in \cite[Theorem 3.2]{sw1} (or in \cite[Section 4.2]{sw2}) that $\Psi$ satisfies the PS-condition. Now a standard minimax argument as in \cite[Theorem 3.2]{sw1} can be employed. Note that with the aid of the vector field $H$ and suitable cutoff functions one can construct a deformation in the usual way as e.g. in \cite{st} (see also \cite{ch}). We leave the details to the reader.

\end{document}